\documentclass[11pt]{amsart}

\usepackage{amsmath}
\usepackage{amssymb}
\usepackage{amsthm}
\usepackage{xcolor}
\usepackage{microtype}
\usepackage{tikz}
\usepackage{pgfplots}
\usetikzlibrary{arrows.meta,calc}
\pgfplotsset{compat=1.18}
\usepackage[
  backend=biber,
  style=alphabetic,
  sorting=nyt,
  giveninits=true,
  maxbibnames=99,
  doi=true,
  url=false,
  isbn=false,
  eprint=true
]{biblatex}
\usepackage[colorlinks=true,linkcolor=blue,citecolor=blue,urlcolor=blue]{hyperref}
\hypersetup{
  pdftitle={Nonuniqueness of Model Potentials and Weak Geodesic Rays},
  pdfauthor={Xia Xiao},
  pdfsubject={K\"ahler geometry and pluripotential theory},
  pdfkeywords={model potentials, non-pluripolar Monge--Amp\`ere measures, multiplier ideals, weak geodesic rays, maximal test curves}
}

\definecolor{caplight}{gray}{0.94}
\definecolor{capmid}{gray}{0.62}
\definecolor{capdark}{gray}{0.28}
\definecolor{orbitblue}{RGB}{52,103,164}
\definecolor{massred}{RGB}{190,48,48}
\tikzset{
  capboundary/.style={draw=black,line width=.75pt},
  capregion/.style={fill=caplight,draw=none},
  cappole/.style={fill=black},
  sphereguide/.style={draw=capmid,line width=.28pt},
  vitalidisk/.style={fill=white,draw=capdark,line width=.58pt},
  vitalipole/.style={fill=black},
  vitaliflow/.style={-{Stealth[length=2.1mm,width=1.5mm]},draw=capdark,line width=.72pt},
  rotorbit/.style={draw=orbitblue,dashed,line width=.8pt},
  masspoint/.style={fill=massred,draw=massred},
  masshalo/.style={draw=massred!70,line width=.55pt}
}

\newcommand{\CC}{\mathbb{C}}
\newcommand{\RR}{\mathbb{R}}
\newcommand{\CP}{\mathbb{CP}}
\newcommand{\cI}{\mathcal{I}}
\newcommand{\OO}{\mathcal{O}}
\newcommand{\ddc}{dd^c}
\newcommand{\one}{\mathbf{1}}
\DeclareMathOperator{\PSH}{PSH}

\theoremstyle{plain}
\newtheorem{thm}{Theorem}[section]
\newtheorem{prop}[thm]{Proposition}
\newtheorem{lem}[thm]{Lemma}

\theoremstyle{definition}
\newtheorem{defn}[thm]{Definition}
\newtheorem{setup}[thm]{Setup}

\newtheoremstyle{mainresult}
  {12pt}
  {12pt}
  {\itshape}
  {}
  {\bfseries}
  {.}
  {.5em}
  {}
\theoremstyle{mainresult}
\newtheorem*{thmA}{Theorem A}
\newtheorem*{thmB}{Theorem B}

\theoremstyle{remark}
\newtheorem{rem}[thm]{Remark}

\title{Nonuniqueness of Model Potentials and Weak Geodesic Rays}
\author{Xia Xiao}
\date{}

\begin{document}

\begin{abstract}
We establish two nonuniqueness phenomena on \(\CP^1\). First, distinct
normalized positive-mass model potentials can have the same non-pluripolar
Monge--Amp\`ere measure. By the contact-set formula, their zero-contact sets
agree \(\omega_{\mathrm{FS}}\)-almost everywhere. Moreover, both have zero
Lelong number at every point and identical multiplier ideal sheaves at every
positive scale. Second, we construct distinct maximal test curves whose
zero-contact sets agree pointwise at every parameter. Their inverse Legendre transforms
are distinct bounded weak geodesic rays from zero with the same pointwise right
initial tangent.
\end{abstract}

\maketitle

\section{Introduction}\label{sec:introduction}

Semmes and Donaldson identified the geodesic equation in the space of
K\"ahler potentials with a homogeneous complex Monge--Amp\`ere equation, and
Chen constructed the corresponding weak geodesic segments
\cite{Semmes92,Donaldson99,Chen00}.  Ross--Witt Nystr\"om later used Legendre
duality to construct weak geodesic rays from curves of singularity types and
to relate them to analytic test configurations \cite{RWN14}; see also
\cite{Darvas17} for the pluripotential theory of weak geodesic rays.  On the
singular side, non-pluripolar products and Monge--Amp\`ere equations with
prescribed singularity type were developed in \cite{BEGZ10,DDL18}.  We use the contact-set formula of Di Nezza--Trapani \cite{DNT21}.

These two themes meet in Questions~5.1 and~5.2 of
\cite[Section~5]{Darvas26}. The first asks whether bounded weak geodesic
segments
\[
 [0,a]\ni t\longmapsto u_t,v_t\in\PSH(X,\omega)\cap L^\infty,
 \qquad u_0=v_0=0,
\]
must agree when their initial tangents satisfy
\(\dot u_0=\dot v_0\). The second concerns uniqueness of positive-mass
\emph{model potentials} from their zero-contact data. We use the Monge--Amp\`ere
measure as the primary datum for this question. Di Nezza--Trapani's
contact-set formula
\cite[Corollary~3.4(ii)]{DNT21} gives, for every normalized
\emph{model potential} \(v\),
\[
 \left\langle(\omega+\ddc v)^n\right\rangle
 =
 \one_{\{v=0\}}\omega^n.
\]
Hence for normalized model potentials \(\phi,\psi\) one has
\[
 \left\langle(\omega+\ddc\phi)^n\right\rangle
 =
 \left\langle(\omega+\ddc\psi)^n\right\rangle
 \quad\Longleftrightarrow\quad
 \one_{\{\phi=0\}}=\one_{\{\psi=0\}}
 \quad \omega^n\text{-a.e.}
\]
Our results give negative answers to both Questions~5.1 and~5.2 on
\(\CP^1\). They also separate measure-theoretic and pointwise contact data. Theorem~A
gives nonuniqueness under almost-everywhere equality of the zero-contact sets,
even after fixing all point Lelong numbers and multiplier ideal sheaves.
Proposition~\ref{prop:maximal-test-curves} gives the pointwise counterpart: it
constructs distinct maximal test curves with the same zero-contact set at every
parameter, and for \(\tau\in(s_0,T]\) their corresponding model slices are
distinct. In particular, any such \(\tau\) gives two distinct model potentials
with exactly the same zero-contact set, yielding a negative answer to
Question~5.2 in its literal pointwise formulation. Their inverse Legendre
transforms yield Theorem~B.

Both constructions take place on
\[
 (\CP^1,\omega_{\mathrm{FS}}),
 \qquad
 \int_{\CP^1}\omega_{\mathrm{FS}}=1.
\]
We work on \(\CP^1\) because the required cap potentials are explicit. The
Vitali construction is local; Remark~\ref{rem:local-construction} isolates the
hypotheses under which the measure-theoretic counterexample extends.

\subsection{Model potentials and Monge--Amp\`ere measures}

\begin{thmA}[Model-potential nonuniqueness]
For every \(\rho\in(0,1)\), there exist distinct normalized positive-mass model
potentials
\[
 \phi,\psi\in\PSH(\CP^1,\omega_{\mathrm{FS}})
\]
and a closed Fubini--Study disk \(C\subset\CP^1\) such that
\[
 \left\langle\omega_{\mathrm{FS}}+\ddc\phi\right\rangle
 =
 \left\langle\omega_{\mathrm{FS}}+\ddc\psi\right\rangle
 =
 \one_C\omega_{\mathrm{FS}},
\]
with total mass \(1-\rho>0\). By the contact-set formula
\cite[Corollary~3.4(ii)]{DNT21}, this is equivalent to
\[
 \one_{\{\phi=0\}}
 =
 \one_{\{\psi=0\}}
 =
 \one_C
 \qquad \omega_{\mathrm{FS}}\text{-a.e.}
\]
Moreover,
\[
 \nu(\phi,x)=\nu(\psi,x)=0
 \qquad\text{for every }x\in\CP^1,
\]
and
\[
 \cI(c\phi)=\OO_{\CP^1}=\cI(c\psi)
 \qquad\text{for every }c>0.
\]
\end{thmA}

Theorem~A is proved in Section~\ref{sec:diffuse} as
Theorem~\ref{thm:principal-b}.

The proof has two steps.  A Vitali packing redistributes the polar mass of the
cap model (Figure~\ref{fig:vitali-packing}) without changing its
non-pluripolar Monge--Amp\`ere measure.  Polar averaging then makes the
residual curvature atomless while keeping it on the minus-infinity locus
(Figure~\ref{fig:rotation-diffusion}).  On a smooth curve every smooth
modification is an isomorphism, so \cite[Proposition~2.2(ii)]{DX24} identifies
the resulting multiplier-ideal singularity type with that of the zero
potential.

\subsection{Weak geodesic rays and initial tangents}

\begin{thmB}[Weak geodesic rays with the same initial tangent]
There exist distinct bounded weak geodesic rays \((u_t)_{t\geq0}\) and
\((v_t)_{t\geq0}\) on \((\CP^1,\omega_{\mathrm{FS}})\) such that
\[
 u_0=v_0=0,
 \qquad
 \dot u_0(x)=\dot v_0(x)
 \quad\text{for every }x\in\CP^1.
\]
Every finite-time slice is bounded, and the complexifications of both rays
solve the homogeneous complex Monge--Amp\`ere equation.
\end{thmB}

Theorem~B is proved in Section~\ref{sec:rays} as
Theorem~\ref{thm:principal-c}. Its key input is
Proposition~\ref{prop:maximal-test-curves}, which gives two distinct maximal
test curves \((\psi_\tau)\) and \((\theta_\tau)\) satisfying
\[
 \{\psi_\tau=0\}=\{\theta_\tau=0\}
 \qquad\text{for every }\tau\in\RR
\]
as subsets of \(\CP^1\). The identity
\[
 \{\dot u_0\geq\tau\}=\{\widehat u_\tau=0\}
\]
then turns this exact contact-set equality into pointwise equality of the
initial tangents. Since the two rays are distinct, restricting them to a finite
interval on which they differ gives distinct bounded weak geodesic segments
with the same initial tangent, and hence a negative answer to Question~5.1.

\subsection*{Structure of the paper}

Section~\ref{sec:preliminaries} recalls model envelopes, Legendre duality, and
the explicit cap model. Section~\ref{sec:atomic} gives the Vitali
redistribution, and Section~\ref{sec:diffuse} proves
Theorem~\ref{thm:principal-b} by polar averaging. Section~\ref{sec:test-curves}
constructs the two maximal test curves used in the ray argument and applies
Legendre duality to prove Theorem~\ref{thm:principal-c}.

\section{Preliminaries}\label{sec:preliminaries}

\subsection{Model envelopes}

\begin{defn}[Model envelopes]\label{defn:model-envelope}
Let \((X,\omega)\) be a compact K\"ahler manifold. For an upper-semicontinuous
function \(f\), set
\[
 P_\omega(f)
 :=
 \left(
 \sup\{v\in\PSH(X,\omega)\mid v\leq f\}
 \right)^*.
\]
For \(\phi\in\PSH(X,\omega)\) with \(\sup_X\phi=0\), define
\[
 P_\omega[\phi]
 :=
 \left(
 \lim_{C\to\infty}P_\omega\bigl(\min\{\phi+C,0\}\bigr)
 \right)^*.
\]
We call \(\phi\) an \emph{\(\omega\)-model potential} if
\(P_\omega[\phi]=\phi\). It has positive mass if its non-pluripolar
Monge--Amp\`ere measure has strictly positive total mass.
\end{defn}

In complex dimension one,
\[
 \left\langle\omega+\ddc u\right\rangle
 =
 \one_{\{u>-\infty\}}(\omega+\ddc u),
\]
so curvature carried by the minus-infinity locus is discarded. With our
normalization, the Lelong number \(\nu(u,x)\) is the mass of the atom of
\(\omega+\ddc u\) at \(x\).

\subsection{Maximal test curves and Legendre duality}

We use the zero-based bounded-ray normalization of
\cite[Theorem~4.1 and Section~5]{Darvas26}, following the primary constructions
\cite{RWN14,Darvas17}.

\begin{defn}[Test curves and maximality]\label{defn:maximal-test-curve}
A test curve is a family
\[
 (\psi_\tau)_{\tau\in\RR},
 \qquad
 \psi_\tau\in\PSH(X,\omega)\cup\{-\infty\},
\]
such that, for every \(x\in X\), the function
\(\tau\mapsto\psi_\tau(x)\) is upper semicontinuous, concave, and decreasing,
while \(\psi_\tau=0\) for all sufficiently negative \(\tau\) and
\(\psi_\tau\equiv-\infty\) for all sufficiently positive \(\tau\).
The test curve is maximal if every finite slice is a normalized model
potential:
\[
 P_\omega[\psi_\tau]=\psi_\tau
 \qquad\text{whenever }\psi_\tau\not\equiv-\infty.
\]
\end{defn}

\begin{thm}[Legendre correspondence]
\label{thm:darvas-correspondence}
In the formulation of \cite[Theorem~4.1]{Darvas26}, building on
\cite{RWN14,Darvas17}, inverse Legendre transform is a bijection from maximal
test curves of
Definition~\ref{defn:maximal-test-curve} to bounded weak geodesic rays
emanating from zero. Explicitly,
\[
 u_t(x)
 =
 \sup_{\tau\in\RR}\bigl(\psi_\tau(x)+t\tau\bigr),
 \qquad t>0,
\]
followed by the canonical upper-semicontinuous representative. The inverse
correspondence is
\[
 \widehat u_\tau(x)
 =
 \inf_{t>0}\bigl(u_t(x)-t\tau\bigr).
\]
\end{thm}

\begin{prop}[Initial-tangent contact sets]
\label{prop:initial-contact}
Let \((u_t)_{t\geq0}\) be a bounded weak geodesic ray emanating from zero and
let \((\widehat u_\tau)_{\tau\in\RR}\) be its Legendre transform.  As
observed in \cite[Section~5]{Darvas26}, for
every \(\tau\in\RR\),
\[
 \{\dot u_0\geq\tau\}
 =
 \{\widehat u_\tau=0\}.
 \tag{2.1}\label{eq:prelim-tangent-contact}
\]
The pointwise initial tangent therefore records the zero-contact filtration
of the dual maximal test curve.
\end{prop}

\subsection{The Fubini--Study cap model}

\begin{setup}[Fubini--Study normalization]\label{setup:sphere}
On the affine chart \(\CC\subset\CP^1\), with coordinate \(w\), let
\[
 \omega_{\mathrm{FS}}
 =
 \ddc\left(\frac12\log(1+|w|^2)\right),
 \qquad
 \int_{\CP^1}\omega_{\mathrm{FS}}=1.
\]
We normalize \(\ddc\) so that
\[
 \ddc\log|w-p|=\delta_p.
\]
\end{setup}

\begin{prop}[One-cap model]\label{prop:cap-models}
For \(0<\rho<1\), define
\[
 b_\rho(w)=
 \begin{cases}
  0,
  & |w|^2\leq (1-\rho)/\rho,\\[4pt]
  \dfrac12
  \left(
   \begin{aligned}
    &-\log(1+|w|^2)
     +(1-\rho)\log|w|^2\\
    &-\rho\log\rho
     -(1-\rho)\log(1-\rho)
   \end{aligned}
  \right),
  & |w|^2\geq (1-\rho)/\rho.
 \end{cases}
\]
Then \(b_\rho\) extends to a normalized positive-mass
\(\omega_{\mathrm{FS}}\)-model potential. Its unique pole is \(\infty\), its
contact region (equivalently, zero-contact set) is
\[
 C_\rho
 =
 \left\{|w|^2\leq\frac{1-\rho}{\rho}\right\},
\]
and
\[
 \omega_{\mathrm{FS}}+\ddc b_\rho
 =
 \one_{C_\rho}\omega_{\mathrm{FS}}+\rho\delta_\infty.
 \tag{2.2}\label{eq:cap-current}
\]
Consequently,
\[
 \left\langle\omega_{\mathrm{FS}}+\ddc b_\rho\right\rangle
 =
 \one_{C_\rho}\omega_{\mathrm{FS}},
 \qquad
 \int_{\CP^1}
 \left\langle\omega_{\mathrm{FS}}+\ddc b_\rho\right\rangle
 =
 1-\rho>0.
\]
\end{prop}

\begin{proof}
Set \(R_\rho^2=(1-\rho)/\rho\). Substitution into the exterior formula gives
zero at \(R_\rho\). If \(x=|w|^2\), the derivative of the exterior expression
with respect to \(\log x\) is
\[
 \frac12\left((1-\rho)-\frac{x}{1+x}\right),
\]
which also vanishes at \(x=R_\rho^2\), so the two pieces join with matching
first radial derivative and produce no boundary measure.

On \(C_\rho\), the potential is zero and its curvature current is
\(\omega_{\mathrm{FS}}\). On the exterior away from infinity,
\(\ddc b_\rho=-\omega_{\mathrm{FS}}\). In the coordinate
\(\zeta=1/w\) near infinity,
\[
 b_\rho=\rho\log|\zeta|+O(1),
\]
so the remaining curvature is \(\rho\delta_\infty\). This proves
\eqref{eq:cap-current}.

The Fubini--Study area of \(\{|w|^2\leq R^2\}\) equals
\(R^2/(1+R^2)\), hence \(C_\rho\) has area \(1-\rho\).
The non-pluripolar product discards the atom at the pole, giving
\(\one_{C_\rho}\omega_{\mathrm{FS}}\). The formula also shows
\(b_\rho\leq0\) with equality exactly on \(C_\rho\). Since the resulting positive-mass measure is supported on
\(\{b_\rho=0\}\), and
\(b_\rho\leq P_{\omega_{\mathrm{FS}}}[b_\rho]\leq0\), it vanishes on
\(\{b_\rho<P_{\omega_{\mathrm{FS}}}[b_\rho]\}\). Hence
\cite[Corollary~3.4(iv)]{DNT21} gives
\[
 P_{\omega_{\mathrm{FS}}}[b_\rho]=b_\rho.
\]
\end{proof}

For later use, introduce the cap-area coordinate
\[
 q=q(w):=\frac{1}{1+|w|^2},
 \qquad q(\infty):=0.
\]
Thus \(q(0)=1\), and \(q(w)\) is the Fubini--Study area of the spherical cap
from the north pole \(\infty\) down to the latitude through \(w\). In
particular,
\[
 U_\rho=\{q<\rho\},
 \qquad
 C_\rho=\{q\geq\rho\}.
\]

\begin{figure}[!htbp]
\centering
\begin{minipage}[c]{0.44\textwidth}
\centering
\begin{tikzpicture}[x=1cm,y=1cm,font=\small]
  \def\R{2.16}
  \def\capY{0.50}
  \def\capRy{0.30}
  \pgfmathsetmacro{\capX}{sqrt(\R*\R-\capY*\capY)}
  \pgfmathsetmacro{\capAngle}{asin(\capY/\R)}
  \path[capregion]
    (-\capX,\capY)
    arc[start angle={180-\capAngle},end angle={\capAngle},radius=\R]
    arc[start angle=0,end angle=-180,x radius=\capX,y radius=\capRy]
    -- cycle;
  \draw[sphereguide] (0,0) ellipse [x radius=.71,y radius=\R];
  \draw[sphereguide] (0,0) ellipse [x radius=\R,y radius=.52];
  \draw[capboundary] (0,0) circle (\R);
  \draw[capmid,densely dashed,line width=.34pt]
    (-\capX,\capY)
    arc[start angle=180,end angle=0,x radius=\capX,y radius=\capRy];
  \draw[capdark,line width=.62pt]
    (-\capX,\capY)
    arc[start angle=180,end angle=360,x radius=\capX,y radius=\capRy];
  \coordinate (p) at (0,\R);
  \fill[cappole] (p) circle (.072);
  \node[anchor=south west,font=\small] at ($(p)+(.10,.035)$)
    {$p_{\mathrm{cap}}=\infty$};
  \coordinate (s) at (0,-\R);
  \fill[cappole] (s) circle (.050);
  \node[anchor=north east] at ($(s)+(-.08,-.03)$) {$0$};
  \node at (-.72,1.22) {$U_\rho$};
  \node at (.58,-.98) {$C_\rho$};
\end{tikzpicture}
\smallskip
\textup{(a) The negative cap is centered at the north pole.}
\end{minipage}
\hfill
\begin{minipage}[c]{0.52\textwidth}
\centering
\begin{tikzpicture}
\def\rhoplot{0.35}
\begin{axis}[
  width=7.15cm,
  height=5.18cm,
  xmin=0,xmax=1,
  ymin=-1.12,ymax=.10,
  axis lines=left,
  axis line style={black,line width=.55pt},
  tick style={black},
  xtick={0,\rhoplot,1},
  xticklabels={$0$,$\rho$,$1$},
  ytick={0},
  yticklabels={$0$},
  xlabel={},
  ylabel={},
  clip=false
]
  \path[fill=caplight] (axis cs:0,-1.12) rectangle (axis cs:\rhoplot,.10);
  \draw[capdark,dashed,line width=.52pt]
    (axis cs:\rhoplot,-1.12)--(axis cs:\rhoplot,.09);
  \addplot[black,line width=1.05pt,domain=.001:\rhoplot,samples=200]
    {0.5*(\rhoplot*ln(x/\rhoplot)+(1-\rhoplot)*ln((1-x)/(1-\rhoplot)))};
  \addplot[black,line width=1.05pt]
    coordinates {(\rhoplot,0) (1,0)};
  \node[anchor=south west] at (axis description cs:.02,.94) {$b_\rho(q)$};
  \node[anchor=west] at (axis description cs:1.01,.09) {$q$};
  \node at (axis cs:.17,-.18) {$U_\rho$};
  \node at (axis cs:.70,-.18) {$C_\rho$};
\end{axis}
\end{tikzpicture}
\smallskip
\textup{(b) Exact radial profile in the cap-area coordinate $q$.}
\end{minipage}
\caption{The one-cap model $b_\rho$. In \textup{(a)}, the shaded region is
exactly the visible part of the negative geodesic cap $U_\rho$, centered at
the north pole $p_{\mathrm{cap}}=\infty$; its projected boundary is an ellipse,
whose hidden half is dashed.  The complementary contact region is $C_\rho$, and
$\omega_{\mathrm{FS}}(U_\rho)=\rho$.  In \textup{(b)},
$q=(1+|w|^2)^{-1}$ is the Fubini--Study cap-area coordinate measured outward
from the north pole, so $U_\rho=\{0\leq q<\rho\}$ and $b_\rho=0$ for
$q\geq\rho$.}
\label{fig:cap}
\end{figure}

\section{Atomic redistribution}\label{sec:atomic}

Fix \(\rho\in(0,1)\), and put
\[
 U_{\mathrm{cap}}=\{b_\rho<0\},
 \qquad
 C=\{b_\rho=0\}=\CP^1\setminus U_{\mathrm{cap}}.
\]
Let \(p_{\mathrm{cap}}\) be the pole of \(b_\rho\). Hence
\[
 \omega_{\mathrm{FS}}+\ddc b_\rho
 =
 \one_C\omega_{\mathrm{FS}}+\rho\delta_{p_{\mathrm{cap}}},
 \qquad
 \left\langle\omega_{\mathrm{FS}}+\ddc b_\rho\right\rangle
 =
 \one_C\omega_{\mathrm{FS}}.
 \tag{3.1}\label{eq:mother-cap-data}
\]

One Vitali redistribution already preserves the non-pluripolar
Monge--Amp\`ere measure.  The selected disks exhaust the negative cap up to
area zero, while the literal zero-contact set acquires an \(\omega_{\mathrm{FS}}\)-null remainder.

\begin{lem}[Single Vitali packing]\label{lem:single-packing}
There exists a pairwise-disjoint family of open Fubini--Study disks
\((B_j)_{j\geq1}\) such that
\[
 \overline{B_j}\subset U_{\mathrm{cap}}\setminus\{p_{\mathrm{cap}}\},
 \qquad
 \omega_{\mathrm{FS}}\left(
 U_{\mathrm{cap}}\setminus\bigcup_{j\geq1}B_j
 \right)=0.
 \tag{3.2}\label{eq:vitali-exhaustion}
\]
Let \(p_j\) be the center of \(B_j\) and put
\(m_j=\int_{B_j}\omega_{\mathrm{FS}}\). Choose
\(G_j\in\mathrm{PSU}(2)\) with \(G_j(\infty)=p_j\). Since \(G_j\) is a
Fubini--Study isometry,
\[
 B_j=G_j(U_{m_j}),
 \qquad
 \varphi_j:=b_{m_j}\circ G_j^{-1}.
\]
We call \(\varphi_j\) the rotated cap block associated with \(B_j\). Then
\[
 u:=\sum_{j\geq1}\varphi_j
\]
converges in \(L^1\) to a normalized positive-mass model potential. If
\[
 K:=\CP^1\setminus\bigcup_{j\geq1}B_j,
\]
then
\[
 \{u=0\}=K,
 \qquad
 C\subset K,
 \qquad
 p_{\mathrm{cap}}\in K\setminus C,
 \qquad
 \omega_{\mathrm{FS}}(K\setminus C)=0.
\]
Moreover,
\[
 \omega_{\mathrm{FS}}+\ddc u
 =
 \one_K\omega_{\mathrm{FS}}
 +\sum_{j\geq1}m_j\delta_{p_j},
 \tag{3.3}\label{eq:single-packing-current}
\]
and
\[
 \left\langle\omega_{\mathrm{FS}}+\ddc u\right\rangle
 =
 \one_K\omega_{\mathrm{FS}}
 =
 \one_C\omega_{\mathrm{FS}}
 \quad\text{as Borel measures}.
 \tag{3.4}\label{eq:single-packing-np}
\]
In particular, \(u(p_{\mathrm{cap}})=0\), whereas
\(b_\rho(p_{\mathrm{cap}})=-\infty\).
\end{lem}

\begin{figure}[!htbp]
\centering
\begin{minipage}[t]{0.72\textwidth}
\centering
\begin{tikzpicture}[x=1cm,y=1cm,font=\small]
  \def\R{2.04}
  \def\capY{0.40}
  \def\capRy{0.28}
  \pgfmathsetmacro{\capX}{sqrt(\R*\R-\capY*\capY)}
  \pgfmathsetmacro{\capAngle}{asin(\capY/\R)}
  \path[capregion]
    (-\capX,\capY)
    arc[start angle={180-\capAngle},end angle={\capAngle},radius=\R]
    arc[start angle=0,end angle=-180,x radius=\capX,y radius=\capRy]
    -- cycle;
  \draw[sphereguide] (0,0) ellipse [x radius=.66,y radius=\R];
  \draw[sphereguide] (0,0) ellipse [x radius=\R,y radius=.47];
  \draw[capboundary] (0,0) circle (\R);
  \draw[capmid,densely dashed,line width=.33pt]
    (-\capX,\capY)
    arc[start angle=180,end angle=0,x radius=\capX,y radius=\capRy];
  \draw[capdark,line width=.60pt]
    (-\capX,\capY)
    arc[start angle=180,end angle=360,x radius=\capX,y radius=\capRy];
  \coordinate (pcap) at (0,\R);
  \fill[cappole] (pcap) circle (.068);
  \node[anchor=south west] at ($(pcap)+(.09,.025)$) {$p_{\mathrm{cap}}$};
  \foreach \x/\y/\rx/\ry in {
    -1.18/.88/.46/.30,
      .00/.96/.40/.27,
      1.18/.88/.46/.30,
    -.78/1.48/.22/.16,
      .78/1.48/.22/.16,
    -1.55/.56/.15/.11,
      .00/.53/.15/.11,
      1.55/.56/.15/.11
  }{
    \filldraw[vitalidisk] (\x,\y) ellipse [x radius=\rx,y radius=\ry];
    \fill[vitalipole] (\x,\y) circle ({min(.027,.13*\rx)});
  }
  \node[anchor=east] at (-1.30,1.62) {$U_{\mathrm{cap}}$};
  \node at (-1.18,.88) {$B_1$};
  \node at (.00,.96) {$B_2$};
  \node at (1.18,.88) {$B_3$};
  \node at (.46,-.93) {$C$};
\end{tikzpicture}
\smallskip
\textup{(a) Selected disjoint subfamily inside the mother cap.}
\end{minipage}
\medskip
\begin{minipage}[t]{0.34\textwidth}
\centering
\begin{tikzpicture}[x=1cm,y=1cm,font=\small]
  \path[use as bounding box] (-1.35,-.92) rectangle (1.35,1.08);
  \filldraw[vitalidisk,line width=.75pt] (0,.08) circle (.90);
  \fill[vitalipole] (0,.08) circle (.066);
  \node at (0,.55) {$B_j$};
  \node[anchor=west] at (.15,.08) {$p_j$};
  \node at (0,-.48) {$\varphi_j<0$};
\end{tikzpicture}
\smallskip
\textup{(b) One cap block.}
\end{minipage}
\hspace{0.055\textwidth}
\begin{minipage}[t]{0.51\textwidth}
\centering
\begin{tikzpicture}[x=1cm,y=1cm,font=\small]
  \path[use as bounding box] (-1.65,-.92) rectangle (2.18,1.08);
  \fill[vitalipole] (-1.28,.18) circle (.090);
  \node[anchor=south] at (-1.28,.40) {$\rho\delta_{p_{\mathrm{cap}}}$};
  \draw[vitaliflow] (-.91,.18)--(-.18,.18);
  \foreach \x/\y/\r in {
    .30/.39/.068,
    .82/.43/.052,
    .48/.02/.058,
    1.10/.05/.042
  }{
    \fill[vitalipole] (\x,\y) circle (\r);
  }
  \node[font=\scriptsize,anchor=north] at (.48,-.34)
    {$\displaystyle
      \sum_{j\ge1}m_j\delta_{p_j}
      \qquad
      \sum_{j\ge1}m_j=\rho$};
\end{tikzpicture}
\smallskip
\textup{(c) Polar-mass redistribution.}
\end{minipage}
\caption{Geometry of Lemma~\ref{lem:single-packing}.  Panel
\textup{(a)} shows a finite sample of the selected disjoint Vitali family,
\textup{(b)} one cap block, and \textup{(c)} the redistribution of the removed
curvature into polar mass.}
\label{fig:vitali-packing}
\end{figure}
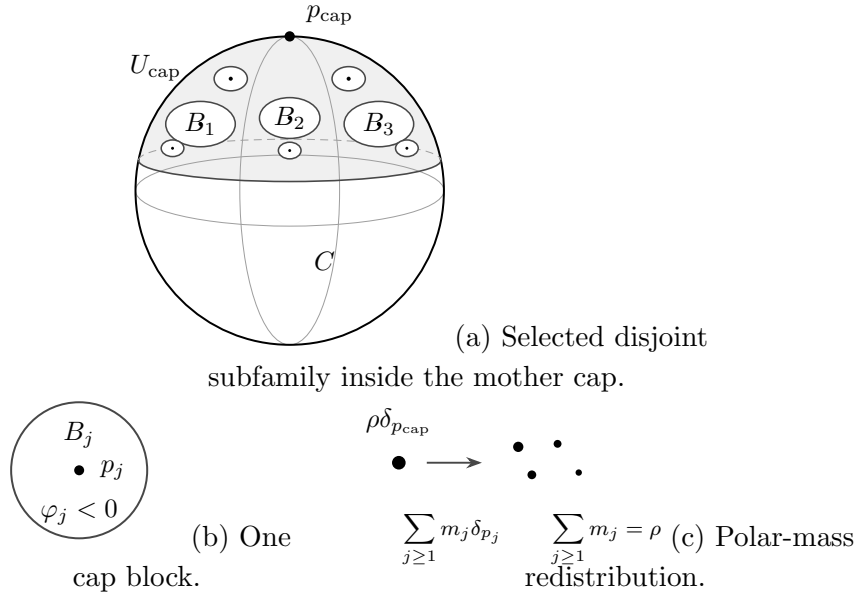

\begin{proof}
The Fubini--Study disks compactly contained in
\(U_{\mathrm{cap}}\setminus\{p_{\mathrm{cap}}\}\) form a Vitali cover of this
punctured cap. The Vitali covering theorem gives a pairwise-disjoint selected
family whose union has full area in
\(U_{\mathrm{cap}}\setminus\{p_{\mathrm{cap}}\}\). Since a point has zero
area, \eqref{eq:vitali-exhaustion} follows and
\[
 \sum_{j\geq1}m_j=\rho.
\]

By the definition of the rotated cap blocks and projective-unitary invariance
of \(\omega_{\mathrm{FS}}\), Proposition~\ref{prop:cap-models} gives
\[
 \varphi_j<0\ \text{on }B_j,
 \qquad
 \varphi_j=0\ \text{on }\CP^1\setminus B_j,
 \qquad
 \ddc\varphi_j
 =-\one_{B_j}\omega_{\mathrm{FS}}+m_j\delta_{p_j}.
\]
For \(N\geq1\), set
\[
 u_N=\sum_{j=1}^N\varphi_j.
\]
Since the sets \(B_j\) are pairwise disjoint,
\[
 \omega_{\mathrm{FS}}+\ddc u_N
 =\one_{\CP^1\setminus\bigcup_{j=1}^N B_j}\omega_{\mathrm{FS}}
 +\sum_{j=1}^N m_j\delta_{p_j}\geq0.
\]
Thus \(u_N\in\PSH(\CP^1,\omega_{\mathrm{FS}})\).  The sequence decreases,
and every \(u_N\) vanishes on the fixed contact region \(C\); in particular,
it does not converge identically to \(-\infty\).  The decreasing-limit theorem
for \(\omega_{\mathrm{FS}}\)-psh functions therefore gives
\[
 u_N\downarrow u\in\PSH(\CP^1,\omega_{\mathrm{FS}}),
\]
and the convergence also holds in \(L^1\).  Passing to distributional
derivatives gives \eqref{eq:single-packing-current}.

Because the negative sets are disjoint, \(\{u=0\}=K\). Since every
\(B_j\) avoids \(p_{\mathrm{cap}}\), one has
\(p_{\mathrm{cap}}\in K\setminus C\), while
\eqref{eq:vitali-exhaustion} gives
\(\omega_{\mathrm{FS}}(K\setminus C)=0\). Every atom in
\eqref{eq:single-packing-current} lies at a pole of \(u\), so the
non-pluripolar product is \(\one_K\omega_{\mathrm{FS}}\), which equals
\(\one_C\omega_{\mathrm{FS}}\) as a Borel measure. This positive-mass measure is supported on \(\{u=0\}\). Since
\(u\leq P_{\omega_{\mathrm{FS}}}[u]\leq0\), it vanishes on
\(\{u<P_{\omega_{\mathrm{FS}}}[u]\}\). Hence
\cite[Corollary~3.4(iv)]{DNT21} gives
\(P_{\omega_{\mathrm{FS}}}[u]=u\). Finally,
\(u(p_{\mathrm{cap}})=0\) because every summand vanishes there.
\end{proof}

Lemma~\ref{lem:single-packing} already gives distinct normalized
positive-mass model potentials \(b_\rho\) and \(u\) with
\[
 \left\langle\omega_{\mathrm{FS}}+\ddc b_\rho\right\rangle
 =
 \left\langle\omega_{\mathrm{FS}}+\ddc u\right\rangle.
\]
Their zero-contact sets are not equal literally, since
\(p_{\mathrm{cap}}\in\{u=0\}\setminus\{b_\rho=0\}\), but they agree
\(\omega_{\mathrm{FS}}\)-almost everywhere. This separates measure-theoretic from literal contact equality.

\begin{rem}[Locality of the curve construction]\label{rem:local-construction}
The Vitali theorem is used only to choose disjoint local blocks whose union
has full \(\omega\)-measure in the region where the curvature is redistributed.
More precisely, let
\((X,\omega)\) be a compact K\"ahler curve and let \(U\subset X\) be open with
\[
 \omega(X\setminus U)>0.
\]
Assume that there is a countable pairwise disjoint family of relatively compact
open sets \(B_j\subset U\) such that
\[
 \omega\!\left(U\setminus\bigcup_j B_j\right)=0.
\]
Choose points \(p_j\in B_j\).  Suppose that for each \(j\) there is a potential
\(\varphi_j\in\PSH(X,\omega)\) satisfying
\[
 \varphi_j=0\quad\text{on }X\setminus B_j,
 \qquad
 \varphi_j<0\quad\text{on }B_j,
 \qquad
 \varphi_j(p_j)=-\infty,
\]
and
\[
 \omega+\ddc\varphi_j
 =
 \one_{X\setminus B_j}\omega+\omega(B_j)\delta_{p_j}.
 \tag{3.8}\label{eq:local-redistribution}
\]

For
\[
 u_N=\sum_{j=1}^N\varphi_j,
\]
disjointness gives
\[
 \omega+\ddc u_N
 =
 \one_{X\setminus\bigcup_{j=1}^N B_j}\omega
 +
 \sum_{j=1}^N\omega(B_j)\delta_{p_j}\geq0.
\]
Hence \(u_N\in\PSH(X,\omega)\). The sequence decreases and is identically
zero on \(X\setminus U\), so its limit is not identically \(-\infty\).
The decreasing-limit theorem for quasi-psh functions therefore gives
\(u\in\PSH(X,\omega)\). Since every summand is nonpositive and
\(u=0\) on \(X\setminus U\), the limit is normalized by
\(\sup_X u=0\). Passing to the limit yields
\[
 \omega+\ddc u
 =
 \one_{X\setminus\bigcup_j B_j}\omega
 +
 \sum_j\omega(B_j)\delta_{p_j}.
\]
Each atom lies in the polar locus of \(u\).  Since the sets \(B_j\) fill
\(U\) up to zero \(\omega\)-measure,
\[
 \left\langle\omega+\ddc u\right\rangle
 =
 \one_{X\setminus U}\omega.
\]
This measure has positive mass and is carried by \(\{u=0\}\). Since
\(u\leq P_\omega[u]\leq0\), it vanishes on \(\{u<P_\omega[u]\}\), and
\cite[Corollary~3.4(iv)]{DNT21} gives
\[
 P_\omega[u]=u.
\]

Once this family has been chosen, the rest of the argument uses only its
disjointness and the condition that its union has full \(\omega\)-measure in
\(U\).  In our Fubini--Study construction, the Vitali theorem provides such a
family.  In general,
\[
 \{u=0\}=X\setminus\bigcup_j B_j,
\]
so this contact set agrees with \(X\setminus U\) only
\(\omega\)-almost everywhere.

On a general K\"ahler curve, the construction of one redistribution block can
be viewed as a local free-boundary problem.  In a coordinate neighborhood of
\(p\), write \(\omega=\ddc\rho\).  One seeks a domain \(B\ni p\), a
constant \(\lambda>0\), and \(C\in\RR\) such that
\[
 \varphi(z)=
 \begin{cases}
  \lambda\log|z-p|-\rho(z)+C,& z\in B,\\
  0,& z\notin B,
 \end{cases}
\]
defines an \(\omega\)-psh function.  To glue the two pieces without creating
a current on \(\partial B\), the natural boundary conditions are
\[
 \lambda\log|z-p|-\rho+C=0,
 \qquad
 d\!\left(\lambda\log|z-p|-\rho\right)=0
 \quad\text{on }\partial B.
\]
Under these conditions, the curvature inside \(B\) is concentrated at the
logarithmic pole, while outside \(B\) it is unchanged.  Thus the existence of
such a block is naturally an obstacle or free-boundary problem.  At the level
of the logarithmic singularity, the corresponding higher-dimensional
expression in holomorphic coordinates \(z=(z_1,\ldots,z_n)\) centered at
\(p\) is
\[
 \log\|z\|_2
 =\frac12\log\!\left(|z_1|^2+\cdots+|z_n|^2\right),
\]
where \(\|\cdot\|_2\) denotes the Euclidean norm in these coordinates.
\end{rem}

\section{Polar averaging and multiplier ideals}\label{sec:diffuse}

Starting from the single-packing model of Lemma~\ref{lem:single-packing}, we
average its residual atoms along polar subsets of rotation orbits.  The
averaged curvature is atomless, remains supported on a polar set, and leaves
the non-pluripolar Monge--Amp\`ere measure unchanged.

\begin{thm}[Model nonuniqueness with identical multiplier ideals]\label{thm:principal-b}
For every \(\rho\in(0,1)\), there exist distinct normalized positive-mass
model potentials
\[
 \phi,\psi\in\PSH(\CP^1,\omega_{\mathrm{FS}}),
\]
and a closed Fubini--Study disk \(C\subset\CP^1\) such that their
non-pluripolar Monge--Amp\`ere measures agree as Borel measures,
\[
 \left\langle\omega_{\mathrm{FS}}+\ddc\phi\right\rangle
 =
 \left\langle\omega_{\mathrm{FS}}+\ddc\psi\right\rangle
 =
 \one_C\omega_{\mathrm{FS}},
 \tag{4.1}\label{eq:multiplier-common-measure}
\]
with total mass \(1-\rho>0\). By
\cite[Corollary~3.4(ii)]{DNT21}, their zero-contact sets agree
\(\omega_{\mathrm{FS}}\)-almost everywhere:
\[
 \one_{\{\phi=0\}}
 =
 \one_{\{\psi=0\}}
 =
 \one_C
 \qquad \omega_{\mathrm{FS}}\text{-a.e.}
 \tag{4.2}\label{eq:multiplier-contact}
\]
Moreover,
\[
 \nu(\phi,x)=\nu(\psi,x)=0
 \qquad\text{for every }x\in\CP^1,
 \tag{4.3}\label{eq:multiplier-zero-lelong}
\]
and
\[
 \cI(c\phi)=\OO_{\CP^1}=\cI(c\psi)
 \qquad\text{for every real }c>0.
 \tag{4.4}\label{eq:multiplier-trivial}
\]
\end{thm}

\begin{proof}
Fix \(\rho\in(0,1)\). Lemma~\ref{lem:single-packing} gives an open cap
\(U_{\mathrm{cap}}\), its contact region
\(C=\CP^1\setminus U_{\mathrm{cap}}\), its center \(p_{\mathrm{cap}}\), and a
normalized positive-mass model potential
\[
 u = \sum_{j\geq1}\varphi_j
\]
such that
\[
 \{u=0\}=K,
 \qquad
 C\subset K,
 \qquad
 \omega_{\mathrm{FS}}(K\setminus C)=0,
 \qquad
 u(p_{\mathrm{cap}})=0,
\]
and
\[
 \omega_{\mathrm{FS}}+\ddc u
 =
 \one_K\omega_{\mathrm{FS}}+\nu
 =
 \one_C\omega_{\mathrm{FS}}+\nu
 \quad\text{as measures},
 \tag{4.5}\label{eq:packed-current}
\]
where
\[
 \nu=\sum_{j\geq1}m_j\delta_{p_j},
 \qquad
 \sum_{j\geq1}m_j=\rho.
 \tag{4.6}\label{eq:atomic-decomposition}
\]
Every \(p_j\) lies in
\(U_{\mathrm{cap}}\setminus\{p_{\mathrm{cap}}\}\), and \(\varphi_j\leq0\)
has a logarithmic pole of weight \(m_j\) at \(p_j\).

Let \(R_\theta\), for
\(\theta\in\RR/(2\pi\mathbb Z)\), be the projective-unitary circle of rotations
about \(p_{\mathrm{cap}}\). It preserves
\(\omega_{\mathrm{FS}}\), \(U_{\mathrm{cap}}\), and \(C\). Its only fixed
points are \(p_{\mathrm{cap}}\) and the antipodal point, which belongs to \(C\).
No \(p_j\) is fixed, and every orbit map
\[
 \theta\longmapsto R_\theta p_j
\]
is injective.

\begin{figure}[t]
\centering
\begin{tikzpicture}[x=1cm,y=1cm]
  \begin{scope}[shift={(3.15,3.55)}]
    \node[font=\small] at (0,1.42) {\textup{(a) Angular support}};
    \draw[capboundary] (0,0) circle (1.02);
    \node[font=\scriptsize] at (-.82,-.83) {$S^1$};
    \draw[rotorbit,line width=.95pt] plot[domain=18:128,samples=100]
      ({1.02*cos(\x)},{1.02*sin(\x)});
    \foreach \t in {18,128}{
      \draw[orbitblue,line width=.95pt]
        ({.92*cos(\t)},{.92*sin(\t)}) -- ({1.12*cos(\t)},{1.12*sin(\t)});
    }
    \node[font=\scriptsize,text=orbitblue] at (.34,1.17) {$J$};
    \foreach \t in {24,27,30,33,40,44,53,56,60,65,76,80,83,93,96,99,108,111,115,121,124}{
      \fill[masspoint] ({1.02*cos(\t)},{1.02*sin(\t)}) circle (.024);
    }
    \node[font=\scriptsize,text=massred] at (.18,.73) {$A$};
  \end{scope}
  \begin{scope}[shift={(0,0)}]
    \node[font=\small] at (0,2.35) {\textup{(b) One atomic pole}};
    \fill[capregion] (0,0) circle (1.55);
    \begin{scope}
      \clip (0,0) circle (1.55);
      \fill[capregion] (-1.55,0.42) rectangle (1.55,1.55);
    \end{scope}
    \draw[capboundary] (0,0) circle (1.55);
    \draw[sphereguide] (0,0) ellipse [x radius=.56,y radius=1.55];
    \draw[sphereguide] (0,0) ellipse [x radius=1.55,y radius=.40];
    \draw[capdark,line width=.55pt] (-1.44,0.42)
      arc[start angle=180,end angle=360,x radius=1.44,y radius=.30];
    \draw[sphereguide,densely dashed] (1.44,0.42)
      arc[start angle=0,end angle=180,x radius=1.44,y radius=.30];
    \fill[cappole] (0,1.55) circle (.058);
    \node[font=\scriptsize,above=1pt] at (0,1.55) {$p_{\mathrm{cap}}$};
    \node[font=\scriptsize] at (-.73,.96) {$U_{\mathrm{cap}}$};
    \node[font=\scriptsize] at (-.68,-.53) {$C$};
    \draw[draw=orbitblue!55,dashed,line width=.72pt]
      (0,.86) ellipse [x radius=1.28,y radius=.19];
    \coordinate (pleft) at (1.01,.98);
    \fill[masspoint] (pleft) circle (.060);
    \node[font=\scriptsize,right=3pt] at (pleft) {$p$};
    \node[font=\footnotesize] at (0,-1.98) {$a\,\delta_p$};
  \end{scope}
  \draw[vitaliflow] (2.05,.15)--(4.25,.15);
  \node[font=\scriptsize,align=center] at (3.15,.72)
    {rotate $p$ by $R_\theta$\\and average, $\theta\sim\sigma$};
  \begin{scope}[shift={(6.3,0)}]
    \node[font=\small] at (0,2.35) {\textup{(c) Diffuse orbit mass}};
    \fill[capregion] (0,0) circle (1.55);
    \begin{scope}
      \clip (0,0) circle (1.55);
      \fill[capregion] (-1.55,0.42) rectangle (1.55,1.55);
    \end{scope}
    \draw[capboundary] (0,0) circle (1.55);
    \draw[sphereguide] (0,0) ellipse [x radius=.56,y radius=1.55];
    \draw[sphereguide] (0,0) ellipse [x radius=1.55,y radius=.40];
    \draw[capdark,line width=.55pt] (-1.44,0.42)
      arc[start angle=180,end angle=360,x radius=1.44,y radius=.30];
    \draw[sphereguide,densely dashed] (1.44,0.42)
      arc[start angle=0,end angle=180,x radius=1.44,y radius=.30];
    \fill[cappole] (0,1.55) circle (.058);
    \node[font=\scriptsize,above=1pt] at (0,1.55) {$p_{\mathrm{cap}}$};
    \node[font=\scriptsize] at (-.73,.96) {$U_{\mathrm{cap}}$};
    \node[font=\scriptsize] at (-.68,-.53) {$C$};
    \draw[draw=orbitblue!55,dashed,line width=.72pt]
      (0,.86) ellipse [x radius=1.28,y radius=.19];
    \foreach \t in {24,27,30,33,40,44,53,56,60,65,76,80,83,93,96,99,108,111,115,121,124}{
      \fill[masspoint] ({1.28*cos(\t)},{.86+.19*sin(\t)}) circle (.022);
    }
    \node[font=\scriptsize,text=massred] at (.15,1.20) {$R_A(p)$};
    \node[font=\footnotesize] at (0,-1.98)
      {$a\,(\theta\mapsto R_\theta p)_*\sigma$};
  \end{scope}
\end{tikzpicture}
\caption{Angular and spatial pictures of the averaging. Panel
\textup{(a)} shows the atomless measure $\sigma$ supported on a Cantor-type
set $A$ contained in an open arc $J\subset S^1$; the red dots are a schematic
sample of this thin support. Panel \textup{(b)} shows the atomic pole $p$ on
its latitude orbit under $R_\theta$, and panel \textup{(c)} shows the diffuse
orbit mass obtained by averaging over $\theta\in A$:
$a\,\delta_p$ is replaced by $a\,(\theta\mapsto R_\theta p)_*\sigma$,
supported on $R_A(p)\subset\CP^1$.}
\label{fig:rotation-diffusion}
\end{figure}
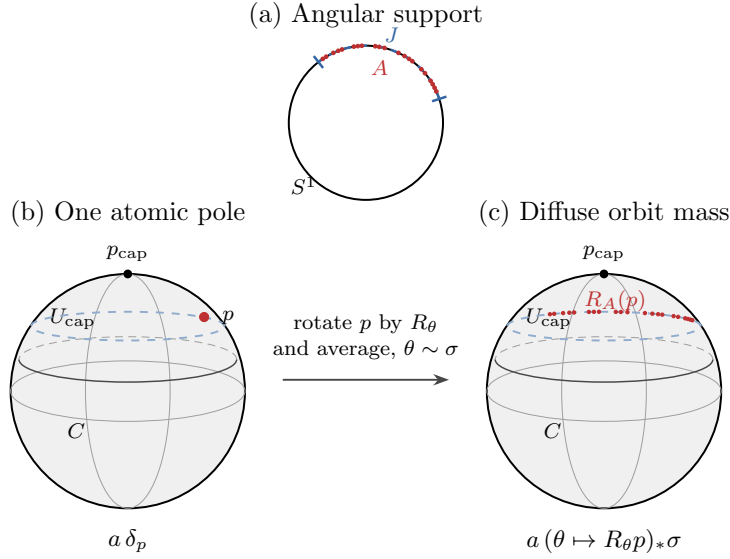

Let \(J\) be a nonempty open arc in
\(\RR/(2\pi\mathbb Z)\). Recursively choose closed arcs \(J_s\), indexed by
finite binary words \(s\), such that the two children \(J_{s0}\) and \(J_{s1}\)
have disjoint closures in the interior of \(J_s\). At level \(n\geq1\), require
\[
 |J_s|\leq \ell_n=\exp(-2^{3n}).
\]
Set
\[
 A=\bigcap_{n\geq1}\bigcup_{|s|=n}J_s,
\]
and let \(\sigma\) be the fair Bernoulli probability on \(A\), characterized by
\[
 \sigma(A\cap J_s)=2^{-n}\qquad (|s|=n).
\]
The measure is atomless, since the masses of the nested cylinders containing
a point tend to zero.

For \(\theta\in A\), let \(J_{s_n(\theta)}\) be its level-\(n\) cylinder.
Chordal distance on the unit circle is at most arc length, so
\[
 \begin{aligned}
 \int_A\log|e^{i\theta}-e^{i\eta}|\,d\sigma(\eta)
 &\leq
 2^{-n}\log\ell_n+(1-2^{-n})\log2\\
 &\leq -2^{2n}+\log2.
 \end{aligned}
 \tag{4.7}\label{eq:polar-log-potential}
\]
Letting \(n\to\infty\) shows that the logarithmic potential is
\(-\infty\) at every point of \(A\). The function
\[
 H(z)=\int_A\log|z-e^{i\eta}|\,d\sigma(\eta)
\]
is subharmonic on \(\CC\) and finite at \(z=0\), where \(H(0)=0\). Since
\(H=-\infty\) on \(e^{iA}\), the set \(e^{iA}\), and hence \(A\) in an angular
chart, is polar. Every nonempty angular arc therefore contains a compact polar set carrying an
atomless probability whose logarithmic potential is \(-\infty\) throughout
its support.

For such a measure \(\sigma\), define
\[
 \Phi_\sigma(x)
 =
 \int u(R_{-\theta}x)\,d\sigma(\theta).
 \tag{4.8}\label{eq:Phi-average}
\]
Projective-unitary invariance and Fubini's theorem give
\[
 \int_{\CP^1}\int
 |u(R_{-\theta}x)|\,d\sigma(\theta)\,
 \omega_{\mathrm{FS}}(x)
 =
 \int_{\CP^1}|u|\,\omega_{\mathrm{FS}}<\infty.
\]
Hence \(\Phi_\sigma\in L^1\).  We also record that the pointwise average
is upper semicontinuous.  Indeed, if \(x_k\to x\), then upper semicontinuity
of \(u\) gives
\[
 \limsup_{k\to\infty}u(R_{-\theta}x_k)
 \leq u(R_{-\theta}x)
 \qquad\text{for every }\theta\in S^1.
\]
Since \(u\leq0\), Fatou's lemma applied to the nonnegative functions
\(-u(R_{-\theta}x_k)\) yields
\[
 \limsup_{k\to\infty}\Phi_\sigma(x_k)
 \leq \Phi_\sigma(x).
\]
Thus \(\Phi_\sigma\) is upper semicontinuous.

Distributional differentiation commutes with the integral. Averaging the
positive currents in \eqref{eq:packed-current} gives
\[
 \omega_{\mathrm{FS}}+\ddc\Phi_\sigma
 =
 \one_C\omega_{\mathrm{FS}}+\nu_\sigma,
 \tag{4.9}\label{eq:diffuse-current}
\]
The right-hand side is a positive current. Together with the upper semicontinuity
just proved, this shows that
\(\Phi_\sigma\in\PSH(\CP^1,\omega_{\mathrm{FS}})\).
Here
\[
 \nu_\sigma
 =
 \sum_{j\geq1}m_j
 \bigl(\theta\longmapsto R_\theta p_j\bigr)_*\sigma.
 \tag{4.10}\label{eq:diffuse-residual}
\]

\begin{rem}[Averaged Dirac measures]\label{rem:averaged-dirac}
For \(p\in\CP^1\), the pushforward measure in
\eqref{eq:diffuse-residual} may equivalently be written
\[
 \mu_p
 :=
 \bigl(\theta\longmapsto R_\theta p\bigr)_*\sigma
 =
 \int_{S^1}\delta_{R_\theta p}\,d\sigma(\theta),
\]
where the last integral is understood weakly.  Namely, for every bounded
continuous function \(f\colon\CP^1\to\RR\),
\[
 \int_{\CP^1} f\,d\mu_p
 =
 \int_{S^1} f(R_\theta p)\,d\sigma(\theta).
\]
The measure \(\mu_p\) is the barycenter of the Dirac masses
\(\delta_{R_\theta p}\).  In this notation,
\eqref{eq:diffuse-residual} simply reads
\[
 \nu_\sigma=\sum_{j\geq1}m_j\mu_{p_j}.
\]
\end{rem}
Every integrand in \eqref{eq:Phi-average} is zero on \(C\) and nonpositive
on \(\CP^1\). Hence
\[
 \Phi_\sigma=0\ \text{on }C,\qquad
 \Phi_\sigma\leq0\ \text{on }\CP^1,\qquad
 \sup_{\CP^1}\Phi_\sigma=0.
 \tag{4.11}\label{eq:diffuse-contact}
\]
Since \(p_{\mathrm{cap}}\) is fixed by every \(R_\theta\) and
\(u(p_{\mathrm{cap}})=0\), one also has
\(\Phi_\sigma(p_{\mathrm{cap}})=0\).  Since
\(p_{\mathrm{cap}}\notin C\), the literal zero-contact set of
\(\Phi_\sigma\) is strictly larger than \(C\); the equality obtained below is
only \(\omega_{\mathrm{FS}}\)-almost everywhere.

Let
\[
 E_\sigma
 =
 \bigcup_{j\geq1}\{R_\theta p_j:\theta\in A\}.
\]
In the affine coordinate used above, the rotations have the form
\(R_\theta(w)=e^{i\theta}w\). Since no \(p_j\) is fixed, its affine
coordinate is nonzero, and its orbit image of \(A\) is a nonzero complex
multiple of \(e^{iA}\). It is therefore polar. Hence the countable union
\(E_\sigma\) is polar, and \eqref{eq:diffuse-residual} shows that
\(\nu_\sigma\) is carried by \(E_\sigma\).

We claim that
\[
 E_\sigma\subset\{\Phi_\sigma=-\infty\}.
 \tag{4.12}\label{eq:E-polar-locus}
\]
Fix \(x=R_\theta p_j\), with \(\theta\in A\). Since all summands in
\eqref{eq:atomic-decomposition} are nonpositive, \(u\leq\varphi_j\).
In a holomorphic coordinate \(z\) centered at \(p_j\), the cap formula in
Proposition~\ref{prop:cap-models} gives
\[
 \varphi_j(y)=m_j\log|z(y)|+O(1).
\]
The orbit crosses \(p_j\) nondegenerately because the point is not fixed.
Compactness of the circle yields a constant \(K_j\) such that
\[
 \varphi_j(R_{-\eta}R_\theta p_j)
 \leq
 m_j\log|e^{i\theta}-e^{i\eta}|+K_j
\]
for every \(\eta\). Integrating against \(\sigma\), applying
\eqref{eq:polar-log-potential} for each \(n\), and then letting
\(n\to\infty\) gives \(\Phi_\sigma(R_\theta p_j)=-\infty\), proving
\eqref{eq:E-polar-locus}.

By \eqref{eq:E-polar-locus}, the measure \(\nu_\sigma\) is carried by the
minus-infinity locus and is therefore discarded by the non-pluripolar product.
Together with \eqref{eq:diffuse-current}, this gives
\[
 \left\langle\omega_{\mathrm{FS}}+\ddc\Phi_\sigma\right\rangle
 =
 \one_C\omega_{\mathrm{FS}}.
 \tag{4.13}\label{eq:diffuse-np}
\]
This measure has mass \(1-\rho>0\) and is supported on
\(C\subset\{\Phi_\sigma=0\}\). Since
\(\Phi_\sigma\leq P_{\omega_{\mathrm{FS}}}[\Phi_\sigma]\leq0\), it vanishes on
\(\{\Phi_\sigma<P_{\omega_{\mathrm{FS}}}[\Phi_\sigma]\}\). Hence
\cite[Corollary~3.4(iv)]{DNT21} gives
\[
 P_{\omega_{\mathrm{FS}}}[\Phi_\sigma]=\Phi_\sigma.
 \tag{4.14}\label{eq:diffuse-model}
\]
Therefore \(\Phi_\sigma\) is a normalized positive-mass model potential. Applying
\cite[Corollary~3.4(ii)]{DNT21} to \(\Phi_\sigma\) and comparing with
\eqref{eq:diffuse-np} yields
\[
 \one_{\{\Phi_\sigma=0\}}=\one_C
 \qquad \omega_{\mathrm{FS}}\text{-a.e.}
\]

The measure \(\nu_\sigma\) is atomless. For fixed \(x\) and \(j\), the
inverse image of \(\{x\}\) under
\(\theta\mapsto R_\theta p_j\) is empty or a singleton. Since \(\sigma\)
has no atoms, each summand of \eqref{eq:diffuse-residual} gives zero mass to
\(\{x\}\), and so does their countable sum. The smooth measure
\(\one_C\omega_{\mathrm{FS}}\) is also atomless. Equation
\eqref{eq:diffuse-current} therefore shows that the Lelong number of
\(\Phi_\sigma\) vanishes at every point.

Following Darvas--Xia, two potentials \(u,v\in\PSH(X,\theta)\) are
\(\cI\)-equivalent when
\[
 \cI(cu)=\cI(cv) \qquad\text{for every }c>0.
\]
The corresponding equivalence classes are called \(\cI\)-singularity types.
By \cite[Proposition~2.2(ii)]{DX24}, this is equivalent to equality of the
Lelong numbers of the pullbacks on every smooth modification: for every
\(\pi:Y\to X\) and every \(y\in Y\),
\[
 \nu(\pi^*u,y)=\nu(\pi^*v,y).
\] Every smooth modification of the smooth compact
curve \(\CP^1\) is an isomorphism. The vanishing just proved therefore gives
\[
 [\Phi_\sigma]_{\cI}=[0]_{\cI},\qquad
 \cI(c\Phi_\sigma)=\cI(0)=\OO_{\CP^1}
 \quad\text{for every }c>0.
 \tag{4.15}\label{eq:trivial-multiplier-ideals}
\]
The convention using \(|f|^2e^{-cu}\) in \cite{DX24} and the alternative
convention using \(|f|^2e^{-2cu}\) differ only by replacing \(c\) with \(2c\).
The assertion for all positive scales is unchanged.

To obtain two distinct averages, observe that the atomic measure \(\nu\) is not
invariant under the rotation circle. Otherwise, an atom of positive mass would
force every point in its infinite orbit to be an atom of the same positive
mass, contradicting finiteness of \(\nu\). No atom occurs at a fixed point by
the construction above.

There is therefore a real continuous function \(f\) on \(\CP^1\) for which
\[
 F(\theta)=\int_{\CP^1}f(R_\theta x)\,d\nu(x)
\]
is nonconstant. Choose open arcs \(J_0,J_1\) so small that the closed intervals
containing the respective ranges of \(F\) are disjoint. In each \(J_i\), apply
the polar construction above to obtain an atomless probability \(\sigma_i\).
Equations \eqref{eq:atomic-decomposition} and
\eqref{eq:diffuse-residual}, followed by Fubini's theorem, give
\[
 \int_{\CP^1}f\,d\nu_{\sigma_i}
 =
 \int F(\theta)\,d\sigma_i(\theta).
\]
The right-hand sides lie in disjoint intervals, so
\(\nu_{\sigma_0}\neq\nu_{\sigma_1}\). Equation
\eqref{eq:diffuse-current} then implies
\(\Phi_{\sigma_0}\neq\Phi_{\sigma_1}\).

Set
\[
 \phi=\Phi_{\sigma_0},\qquad \psi=\Phi_{\sigma_1}.
\]
Equations \eqref{eq:diffuse-np}, \eqref{eq:diffuse-model}, and
\eqref{eq:trivial-multiplier-ideals} show that the two potentials are
normalized model potentials with the same non-pluripolar measure
\(\one_C\omega_{\mathrm{FS}}\) of mass \(1-\rho>0\), and their contact-set
characteristic functions both agree with \(\one_C\)
\(\omega_{\mathrm{FS}}\)-almost everywhere. They also satisfy
\[
 \nu(\phi,x)=\nu(\psi,x)=0
 \qquad\text{for every }x\in\CP^1,
\]
together with
\[
 \cI(c\phi)=\OO_{\CP^1}=\cI(c\psi)
 \qquad(c>0).
\]
They are distinct because \(\nu_{\sigma_0}\neq\nu_{\sigma_1}\).
\end{proof}

\section{From contact filtrations to weak geodesic rays}
\label{sec:test-curves}

\subsection{Equal contact filtrations}

The ray problem requires literal contact data. Initial-tangent superlevel sets
are identified pointwise with test-curve contact sets, so these sets must agree
at every dual parameter.  The single-packing potential from
Lemma~\ref{lem:single-packing} may have additional zeroes on a null set, but
the perturbation below pushes those zeroes strictly below zero while changing
the residual atomic measure.  Thus the zero-contact filtration remains
unchanged, and no second mixing construction is needed.

\begin{prop}[Distinct maximal test curves with identical contact filtrations]\label{prop:maximal-test-curves}
Let \(\omega_{\mathrm{FS}}\) be normalized as in
Setup~\ref{setup:sphere}, and let \(b_\rho\) be the cap model of
Proposition~\ref{prop:cap-models}. There exist
\[
 0<\delta<s_0^2<T^2<1,\qquad \epsilon>0,
\]
and distinct maximal test curves
\((\psi_\tau)_{\tau\in\RR}\) and
\((\theta_\tau)_{\tau\in\RR}\) on
\((\CP^1,\omega_{\mathrm{FS}})\) whose zero-contact sets agree literally for
every \(\tau\). More precisely,
\[
 \psi_\tau=\theta_\tau=0\quad(\tau\leq0),\qquad
 \psi_\tau=b_{\tau^2}\quad(0<\tau\leq T),
\]
\[
 \psi_\tau=\theta_\tau=-\infty\quad(\tau>T),
\]
and
\[
 \theta_\tau=b_{\tau^2}\quad(0<\tau\leq s_0).
\]
For \(s_0<\tau\leq T\), the function \(\theta_\tau\) is a normalized
positive-mass model potential satisfying
\[
 \{\theta_\tau=0\}=\{b_{\tau^2}=0\}=C_{\tau^2}.
\]
For every \(x\in\CP^1\), both parameter functions
\(\tau\mapsto\psi_\tau(x)\) and
\(\tau\mapsto\theta_\tau(x)\) are upper semicontinuous, concave, and
decreasing. The two curves are distinct for every
\(\tau\in(s_0,T]\).
\end{prop}

\begin{proof}
Choose
\[
0<\delta<s_0^2<T^2<1.
\]
By Proposition~\ref{prop:cap-models},
\[
 \omega_{\mathrm{FS}}+\ddc b_\rho
 =
 \one_{C_\rho}\omega_{\mathrm{FS}}+\rho\delta_\infty.
 \tag{5.1}\label{eq:cap-family-current}
\]
Apply Lemma~\ref{lem:single-packing} with parameter \(\rho=\delta\).
Write \(h\) for the resulting single-packing potential, \(K\) for its
zero-contact set, and \(\nu\) for its residual atomic measure. Then
\[
 h\leq0,
 \qquad
 h=0\quad\text{on }C_\delta,
 \qquad
 \{h=0\}=K\supset C_\delta,
\]
where \(K\setminus C_\delta\) has zero Fubini--Study area. Moreover,
\[
 \omega_{\mathrm{FS}}+\ddc h
 =\one_K\omega_{\mathrm{FS}}+\nu
 =\one_{C_\delta}\omega_{\mathrm{FS}}+\nu
 \quad\text{as currents},
 \tag{5.2}\label{eq:h-current}
\]
\(\nu\) is a positive atomic measure of total mass \(\delta\), every atom
of \(\nu\) is a pole of \(h\), and
\[
 h(\infty)=0,
 \qquad
 \nu(\{\infty\})=0.
\]
The possible additional zeroes in \(K\setminus C_\delta\) will disappear
in the perturbed slices.

Let
\[
 P:=\{h=-\infty\}\subset U_\delta,
\]
and let \(S\subset P\) be the countable set of atoms of \(\nu\). On
\(\CP^1\setminus\{\infty\}\), put
\[
 d=h-b_\delta.
\]
Then \(d=0\) on \(C_\delta\), while \(d=-\infty\) on \(P\). At every point
where \(d\) is finite, \(h\leq0\) gives
\[
 d\leq-b_\delta.
 \tag{5.3}\label{eq:d-bound}
\]
Write \(d_+=\max\{d,0\}\) on this finite locus.  The positive part \(d_+\)
is the only part of the perturbation that can oppose the decrease and
concavity of the cap family.

For \(s_0<\tau\leq T\), set
\[
 \lambda_\tau=\epsilon(\tau-s_0)^2,
 \qquad
 q_\tau=b_{\tau^2}+\lambda_\tau d
 \tag{5.4}\label{eq:q-definition}
\]
on \(\CP^1\setminus\{\infty\}\), and take the upper-semicontinuous extension
at infinity.  We shall choose \(\epsilon>0\) below.

\medskip
\noindent
\emph{Step 1: positivity of the slices.}
Subtracting \eqref{eq:cap-family-current} at \(\rho=\delta\) from
\eqref{eq:h-current} gives
\[
 \omega_{\mathrm{FS}}+\ddc q_\tau
 =
 \one_{C_{\tau^2}}\omega_{\mathrm{FS}}
 +(\tau^2-\delta\lambda_\tau)\delta_\infty
 +\lambda_\tau\nu.
 \tag{5.5}\label{eq:q-current}
\]
The final choice of \(\epsilon\) in \eqref{eq:epsilon-choice} will include
\[
 \epsilon(T-s_0)^2<1.
\]
Consequently \(0\leq\lambda_\tau<1\), and
\[
 \tau^2-\delta\lambda_\tau
 >\tau^2-\delta>0.
\]
Hence the current in \eqref{eq:q-current} is positive.  On the affine chart,
\(b_{\tau^2}\) and \(b_\delta\) are continuous and \(h\) is upper
semicontinuous, so the function in \eqref{eq:q-definition} is upper
semicontinuous there. With the chosen extension at infinity, it follows that
\[
 q_\tau\in\PSH(\CP^1,\omega_{\mathrm{FS}}).
\]

\medskip
\noindent
\emph{Step 2: contact set and modelity.}
Since \(\tau^2>\delta\),
\[
 C_{\tau^2}\subset C_\delta,
 \qquad
 U_\delta\subset U_{\tau^2}.
\]
On \(C_{\tau^2}\), both \(b_{\tau^2}\) and \(d\) vanish, hence
\(q_\tau=0\).

Let \(x\in U_{\tau^2}\) be a finite point at which \(d(x)>-\infty\).
Since \(h\leq0\), \eqref{eq:d-bound} gives \(d(x)\leq-b_\delta(x)\).  Since
the cap family decreases with its parameter and \(\tau^2>\delta\),
\(b_{\tau^2}(x)\leq b_\delta(x)\).  Therefore
\[
 \begin{aligned}
 q_\tau(x)
 &\leq b_{\tau^2}(x)-\lambda_\tau b_\delta(x)\\
 &\leq(1-\lambda_\tau)b_{\tau^2}(x)<0.
 \end{aligned}
\]
This estimate includes the possible extra zeros in
\((K\setminus C_\delta)\setminus\{\infty\}\).  On \(P\), one has
\(q_\tau=-\infty\).  At infinity, the logarithmic coefficient is
\(\tau^2-\delta\lambda_\tau>0\), so the upper-semicontinuous extension also
satisfies \(q_\tau(\infty)=-\infty\).  Hence
\[
 \{q_\tau=0\}=C_{\tau^2}
 \tag{5.6}\label{eq:q-contact}
\]
as a literal set.

The atoms of \eqref{eq:q-current} lie on the \(-\infty\)-locus: the atoms of
\(\nu\) are poles of \(h\), while the coefficient of \(\delta_\infty\) is
positive and \(q_\tau\) has a logarithmic pole at infinity. Consequently,
\[
 \left\langle
 \omega_{\mathrm{FS}}+\ddc q_\tau
 \right\rangle
 =
 \one_{C_{\tau^2}}\omega_{\mathrm{FS}}.
 \tag{5.7}\label{eq:q-np}
\]
This measure has mass \(1-\tau^2>0\) and is supported on \(\{q_\tau=0\}\).
Since \(q_\tau\leq P_{\omega_{\mathrm{FS}}}[q_\tau]\leq0\), it vanishes on
\(\{q_\tau<P_{\omega_{\mathrm{FS}}}[q_\tau]\}\). Therefore
\cite[Corollary~3.4(iv)]{DNT21} gives
\[
 P_{\omega_{\mathrm{FS}}}[q_\tau]=q_\tau.
 \tag{5.8}\label{eq:q-model}
\]
Thus every \(q_\tau\) is a normalized positive-mass model potential.

\medskip
\noindent
\emph{Step 3: decreasingness and concavity in \(\tau\).}
For a finite affine point \(w\), put
\[
 a=(1+|w|^2)^{-1}.
\]
When \(0<a<\rho<1\), the cap formula becomes
\[
 b_\rho(w)=-\frac12D(\rho\Vert a),
\]
where
\[
 D(r\Vert a)
 =
 r\log\frac ra+(1-r)\log\frac{1-r}{1-a}.
\]
For \(a<\tau^2\), set
\[
 L_\tau(a)
 =
 \log\left(\frac{\tau^2(1-a)}{(1-\tau^2)a}\right)>0.
\]
Differentiation gives
\[
 \partial_\tau b_{\tau^2}(w)=-\tau L_\tau(a),
 \qquad
 \partial_\tau^2 b_{\tau^2}(w)
 =
 -L_\tau(a)-\frac{2}{1-\tau^2}.
 \tag{5.9}\label{eq:cap-tau-derivatives}
\]
When \(\tau^2\leq a\), one has \(b_{\tau^2}(w)=0\). At
\(\tau=\sqrt a\), the active formula and its first derivative both vanish.
Hence \(\tau\mapsto b_{\tau^2}(w)\) is upper semicontinuous, decreasing, and
concave on \((0,T]\). At infinity it is identically \(-\infty\) there.

Only points with \(d(w)>0\) require an estimate. Such a point lies in
\(U_\delta\setminus P\), hence \(0<a<\delta\), and
\eqref{eq:d-bound} gives
\[
 0<d(w)\leq d_+(w)\leq-b_\delta(w)
 =\frac12D(\delta\Vert a).
 \tag{5.10}\label{eq:d-positive-bound}
\]
There is a constant \(M>0\) such that, for
\(0<a<\delta\) and \(s_0\leq\tau\leq T\),
\[
 D(\delta\Vert a)\leq M\tau L_\tau(a),
 \qquad
 D(\delta\Vert a)
 \leq M\left(L_\tau(a)+\frac{2}{1-\tau^2}\right).
\]
Indeed, the two corresponding ratios are continuous in the interior.  As
\(a\to\delta^-\), their numerators tend to zero while their denominators
remain positive because \(s_0^2>\delta\).  As \(a\to0^+\),
\[
 D(\delta\Vert a)=\delta\log(1/a)+O(1),
 \qquad
 L_\tau(a)=\log(1/a)+O(1),
\]
uniformly for \(\tau\in[s_0,T]\), so both ratios remain bounded.

Choose \(\epsilon>0\) such that
\[
 \epsilon(T-s_0)^2<1,
 \qquad
 \epsilon M(T-s_0)\leq1,
 \qquad
 \epsilon M\leq1.
 \tag{5.11}\label{eq:epsilon-choice}
\]
Fix a finite point \(w\in U_\delta\setminus P\).  Since
\(0<a<\delta<s_0^2\leq\tau^2\), the active cap formula applies throughout
\([s_0,T]\).  Equations \eqref{eq:q-definition} and
\eqref{eq:cap-tau-derivatives} give
\[
 \partial_\tau q_\tau
 =
 -\tau L_\tau(a)+2\epsilon(\tau-s_0)d,
\]
\[
 \partial_\tau^2q_\tau
 =
 -L_\tau(a)-\frac{2}{1-\tau^2}+2\epsilon d.
 \tag{5.12}\label{eq:q-tau-derivatives}
\]
Since \(d\leq d_+\),
\[
 \partial_\tau q_\tau
 \leq
 -\tau L_\tau(a)+2\epsilon(\tau-s_0)d_+,
\]
and
\[
 \partial_\tau^2q_\tau
 \leq
 -L_\tau(a)-\frac{2}{1-\tau^2}+2\epsilon d_+.
\]
The bound \eqref{eq:d-positive-bound} and the choice
\eqref{eq:epsilon-choice} imply
\[
 \begin{aligned}
 2\epsilon(\tau-s_0)d_+
 &\leq \epsilon(T-s_0)D(\delta\Vert a)\\
 &\leq \tau L_\tau(a),
 \end{aligned}
\]
and
\[
 2\epsilon d_+
 \leq
 \epsilon D(\delta\Vert a)
 \leq
 L_\tau(a)+\frac{2}{1-\tau^2}.
\]
Hence both derivatives in \eqref{eq:q-tau-derivatives} are nonpositive.
Outside \(U_\delta\), one has \(d=0\), and the conclusion follows directly
from the cap family.

\medskip
\noindent
\emph{Step 4: assembly of the test curves.}
At \(\tau=s_0\), both \(\lambda_\tau\) and \(\lambda_\tau'\) vanish. At
every finite point outside \(P\), the branch \(q_\tau\) therefore joins
\(b_{\tau^2}\) with matching value and first derivative. At a point of
\(P\), the parameter function equals the finite concave function
\(b_{\tau^2}\) through \(s_0\) and equals \(-\infty\) afterward; this
extended-real-valued function is decreasing, concave, and upper
semicontinuous. At infinity, the parameter function is zero for \(\tau\leq0\)
and \(-\infty\) for \(\tau>0\), with the same properties.

Define
\[
 \psi_\tau=
 \begin{cases}
  0,&\tau\leq0,\\
  b_{\tau^2},&0<\tau\leq T,\\
  -\infty,&\tau>T,
 \end{cases}
\]
and
\[
 \theta_\tau=
 \begin{cases}
  0,&\tau\leq0,\\
  b_{\tau^2},&0<\tau\leq s_0,\\
  q_\tau,&s_0<\tau\leq T,\\
  -\infty,&\tau>T.
 \end{cases}
\]
The preceding calculations give pointwise decreasingness, concavity, and upper
semicontinuity. At \(\tau=0\), every finite affine point has a zero plateau for
small positive \(\tau\), while infinity jumps from zero to \(-\infty\); both
behaviors are upper semicontinuous and extended-real concave. Extending by
\(-\infty\) after \(T\) preserves these properties because the finite
effective domain ends at \(T\).

Every finite slice is a normalized model potential: this follows from
Proposition~\ref{prop:cap-models} for the cap slices and from
\eqref{eq:q-model} for the perturbed slices. Hence both families are maximal
test curves by Definition~\ref{defn:maximal-test-curve}. Their contact sets
are
\[
 \CP^1\quad(\tau\leq0),
 \qquad
 C_{\tau^2}\quad(0<\tau\leq T),
 \qquad
 \varnothing\quad(\tau>T),
\]
so they agree literally for every \(\tau\).

Finally, \(\nu\) has total mass \(\delta>0\), so \(S\neq\varnothing\). If
\(p\in S\) and \(\tau\in(s_0,T]\), then
\[
 \psi_\tau(p)=b_{\tau^2}(p)>-\infty,
 \qquad
 \theta_\tau(p)=q_\tau(p)=-\infty,
\]
because \(p\neq\infty\). Therefore the two maximal test curves are distinct
at every such parameter.
\end{proof}

\subsection{Weak geodesic rays with equal initial tangent}\label{sec:rays}

Legendre duality carries the contact-set identity of
Proposition~\ref{prop:maximal-test-curves} to equality of the pointwise
initial tangents.

\begin{thm}[Distinct weak geodesic rays with the same initial tangent]
\label{thm:principal-c}
There exist distinct bounded weak geodesic rays on
\((\CP^1,\omega_{\mathrm{FS}})\),
\[
 (u_t)_{t\geq0},\qquad (v_t)_{t\geq0},
\]
such that \(u_0=v_0=0\).
Every finite-time slice \(u_t\) and \(v_t\) is bounded. Their complexifications
are \(\pi^*\omega_{\mathrm{FS}}\)-plurisubharmonic and solve the homogeneous
complex Monge--Amp\`ere equation. The pointwise right derivatives
\[
 \dot u_0(x)=\lim_{t\to0^+}\frac{u_t(x)}{t},
 \qquad
 \dot v_0(x)=\lim_{t\to0^+}\frac{v_t(x)}{t}
\]
exist and satisfy
\[
 \dot u_0(x)=\dot v_0(x)\qquad\text{for every }x\in\CP^1.
\]
\end{thm}

\begin{proof}
Let \((\psi_\tau)\) and \((\theta_\tau)\) be the distinct maximal test curves
from Proposition~\ref{prop:maximal-test-curves}. They vanish for
\(\tau\leq0\), have normalized model slices for \(0<\tau\leq T\), and equal
\(-\infty\) for \(\tau>T\).

For \(t>0\), define their inverse Legendre transforms by
\[
 u_t(x)=\sup_{\tau\in\RR}\bigl(\psi_\tau(x)+t\tau\bigr),
 \qquad
 v_t(x)=\sup_{\tau\in\RR}\bigl(\theta_\tau(x)+t\tau\bigr),
 \tag{5.13}\label{eq:inverse-legendre}
\]
and take the canonical upper-semicontinuous representatives. By Theorem~\ref{thm:darvas-correspondence}, these inverse Legendre transforms
are bounded weak geodesic rays
with \(u_0=v_0=0\). Their complexifications are
\(\pi^*\omega_{\mathrm{FS}}\)-plurisubharmonic and solve the homogeneous
complex Monge--Amp\`ere equation.

Boundedness at every finite time also follows directly from
\eqref{eq:inverse-legendre}. The slice at \(\tau=0\) is zero, every finite
test-curve slice is nonpositive, and no finite slice occurs for \(\tau>T\).
Therefore
\[
 0\leq u_t\leq Tt,\qquad
 0\leq v_t\leq Tt
 \tag{5.14}\label{eq:ray-bounds}
\]
for every finite \(t\).

Proposition~\ref{prop:initial-contact} gives, for every
\(\tau\in\RR\),
\[
 \{\dot u_0\geq\tau\}=\{\widehat u_\tau=0\},
 \qquad
 \{\dot v_0\geq\tau\}=\{\widehat v_\tau=0\}.
 \tag{5.15}\label{eq:tangent-contact}
\] The bijection identifies
\(\widehat u_\tau=\psi_\tau\) and
\(\widehat v_\tau=\theta_\tau\). Proposition
\ref{prop:maximal-test-curves} gives
\[
 \{\psi_\tau=0\}=\{\theta_\tau=0\}
 \qquad\text{for every }\tau\in\RR.
\]
Combining this with \eqref{eq:tangent-contact} yields
\[
 \{\dot u_0\geq\tau\}=\{\dot v_0\geq\tau\}
 \qquad\text{for every }\tau\in\RR
\]
as literal subsets of \(\CP^1\). Equality of all real superlevel sets gives
\[
 \dot u_0(x)=\dot v_0(x)
 \qquad\text{for every }x\in\CP^1.
\]

The test curves are distinct by Proposition
\ref{prop:maximal-test-curves}.  Bijectivity of the Legendre correspondence
then implies that the rays are distinct.  Since they agree at \(t=0\),
\(u_t\neq v_t\) for at least one \(t>0\).
\end{proof}

\section*{Acknowledgments}
The author would like to thank his supervisors, Julien Keller and Hugues Auvray, for
their guidance, encouragement, and many helpful discussions.  This work was
supported by the FRQNT grant \emph{M\'etriques K\"ahl\'eriennes sp\'eciales
singuli\`eres et non-compactes} (DOI: 10.69777/343263).  A large language model
was used to assist with algebraic checks, selected explicit calculations, and
language editing.  The author independently checked all mathematical arguments
and conclusions.

\printbibliography

@article{BEGZ10,
  author       = {Boucksom, S{\'e}bastien and Eyssidieux, Philippe and Guedj, Vincent and Zeriahi, Ahmed},
  title        = {Monge--Amp\`ere equations in big cohomology classes},
  journaltitle = {Acta Math.},
  year         = {2010},
  volume       = {205},
  number       = {2},
  pages        = {199--262},
  doi          = {10.1007/s11511-010-0054-7},
  shorthand    = {BEGZ10},
}

@article{Chen00,
  author       = {Chen, Xiuxiong},
  title        = {The space of K\"ahler metrics},
  journaltitle = {J. Differential Geom.},
  year         = {2000},
  volume       = {56},
  number       = {2},
  pages        = {189--234},
  doi          = {10.4310/jdg/1090347643},
  shorthand    = {Chen00},
}

@article{DDL18,
  author       = {Darvas, Tam{\'a}s and Di Nezza, Eleonora and Lu, Chinh H.},
  title        = {Monotonicity of non-pluripolar products and complex Monge--Amp\`ere equations with prescribed singularity},
  journaltitle = {Anal. PDE},
  year         = {2018},
  volume       = {11},
  number       = {8},
  pages        = {2049--2087},
  doi          = {10.2140/apde.2018.11.2049},
  shorthand    = {DDL18},
}

@article{Darvas17,
  author       = {Darvas, Tam{\'a}s},
  title        = {Weak geodesic rays in the space of K\"ahler potentials and the class {$\mathcal{E}(X,\omega)$}},
  journaltitle = {J. Inst. Math. Jussieu},
  year         = {2017},
  volume       = {16},
  number       = {4},
  pages        = {837--858},
  doi          = {10.1017/S1474748015000316},
  shorthand    = {Darvas17},
}

@online{Darvas26,
  author       = {Darvas, Tam{\'a}s},
  title        = {A decade of metric geometry in the space of K\"ahler metrics},
  date         = {2026-04-21},
  eprint       = {2604.18981},
  eprinttype   = {arxiv},
  eprintclass  = {math.DG},
  doi          = {10.48550/arXiv.2604.18981},
  shorthand    = {Darvas26},
}

@article{DX24,
  author       = {Darvas, Tam{\'a}s and Xia, Mingchen},
  title        = {The volume of pseudoeffective line bundles and partial equilibrium},
  journaltitle = {Geom. Topol.},
  year         = {2024},
  volume       = {28},
  number       = {4},
  pages        = {1957--1993},
  doi          = {10.2140/gt.2024.28.1957},
  shorthand    = {DX24},
}

@article{DNT21,
  author       = {Di Nezza, Eleonora and Trapani, Stefano},
  title        = {Monge--Amp\`ere measures on contact sets},
  journaltitle = {Math. Res. Lett.},
  year         = {2021},
  volume       = {28},
  number       = {5},
  pages        = {1337--1352},
  doi          = {10.4310/MRL.2021.v28.n5.a3},
  shorthand    = {DNT21},
}

@incollection{Donaldson99,
  author       = {Donaldson, Simon K.},
  title        = {Symmetric spaces, K\"ahler geometry and Hamiltonian dynamics},
  booktitle    = {Northern California Symplectic Geometry Seminar},
  series       = {American Mathematical Society Translations, Series 2},
  volume       = {196},
  publisher    = {American Mathematical Society},
  location     = {Providence, RI},
  year         = {1999},
  pages        = {13--33},
  doi          = {10.1090/trans2/196/02},
  shorthand    = {Donaldson99},
}

@article{RWN14,
  author       = {Ross, Julius and Witt Nystr\"om, David},
  title        = {Analytic test configurations and geodesic rays},
  journaltitle = {J. Symplectic Geom.},
  year         = {2014},
  volume       = {12},
  number       = {1},
  pages        = {125--169},
  doi          = {10.4310/JSG.2014.v12.n1.a5},
  shorthand    = {RWN14},
}

@article{Semmes92,
  author       = {Semmes, Stephen},
  title        = {Complex Monge--Amp\`ere and symplectic manifolds},
  journaltitle = {Amer. J. Math.},
  year         = {1992},
  volume       = {114},
  number       = {3},
  pages        = {495--550},
  doi          = {10.2307/2374768},
  shorthand    = {Semmes92},
}

\end{document}